\documentclass[11pt]{article}
\usepackage{oldlfont}
\usepackage{tikz}

\newtheorem{theorem}{Theorem}
\newtheorem{proposition}[theorem]{Proposition}
\newtheorem{lemma}[theorem]{Lemma}
\newtheorem{remark}[theorem]{Remark}
\newtheorem{corollary}[theorem]{Corollary}

\def\newpic#1{%
   \def\emline##1##2##3##4##5##6{%
      \put(##1,##2){\special{em:point #1##3}}%
      \put(##4,##5){\special{em:point #1##6}}%
      \special{em:line #1##3,#1##6}}}
\newpic{}

\begin{document}

\title{Some Aspects of the Wiener Index for Sun Graphs}

\author{ Mohamed Amine Boutiche
%\thanks{Supported by \dots}
\\
USTHB, Faculty of Mathematics, LaROMaD Laboratory \\
BP 32 El Alia, 16111 Bab Ezzouar, Algiers, Algeria
\\ {\tt mboutiche@usthb.dz} }

\date{}

\maketitle
\begin{abstract}
The Wiener index $W(G)$ is the sum of distances of all pairs of
vertices of the graph $G$. The Wiener polarity index $W_{p}(G)$ of a
graph $G$ is the number of unordered pairs of vertices $u$ and $v$
of $G$ such that the distance $d_{G}(u,v)$ between $u$ and $v$ is
$3$. In this paper the Wiener and the Wiener polarity indices of sun
graphs are computed. A relationship between those indices with some
other topological indices are presented. Finally, we find the Hosoya
(Wiener) polynomial for sun graphs.
\end{abstract}

\maketitle


%% or \MSC[2008] code \sep 05C35 \sep 05C50 \sep 05C90. (2000 is the default)

\maketitle

%%%%%%%%%%%%%%%%%%%%%%%%%%%%%%%%%%%%%

\section{Introduction}
In theoretical chemistry, distance based molecular structure
descriptors are used for modeling physical, pharmacologic, biologic
and others properties of chemical compounds \cite{Tri93,Wie47}.

Let $G=(V,E)$ be a finite undirected connected graph. The distance
between two vertices $u$ and $v$ in $G$, denoted by $d_{G}(u,v)$, is
the length of a shortest path between $u$ and $v$ in $G$. A tree is
a connected acyclic graph. Let $d_{G}(v)$ of a vertex $v$ be the sum
of all distances between $v$ and all other vertices of $G$.

The Wiener index $W(G)$ of $G$ is defined by

\begin{center}
$W(G)=\displaystyle \sum_{u,v\subseteq V(G)} d_{G}(u,v)$.
\end{center}

One can define the Wiener index also in a slightly different way
\begin{center} $W(G)=\displaystyle \frac{1}{2} \sum_{v\in V(G)} d_{G}(v)$.\end{center}

The Wiener polarity index $W_{P}(G)$ of a graph $G$ of order $n$ is
defined as the number of unordered pairs of vertices $u$ and $v$ of
$G$ such that the distance $d_{G}(u,v)$ between $u$ and $v$ is $3$.
$W_{P}(G)$ was recently introduced and receive more attention. The
authors in \cite{LL98} demonstrated quantitative structure-property
relationships in a series of acyclic and cycle-containing
hydrocarbons by using $W_{P}(G)$. Hosoya \cite{Hos02} found a
physical-chemical interpretation of $W_{P}(G)$. Recently,
\cite{DLS09} gives all trees (resp. unicyclic graphs) of order $n$
minimizing and maximizing the Wiener polarity. Some extremal
properties on particular trees were given in
\cite{DXT10,Deng11,DX10,LHH10}.

A $k-$sun graph $(k \geq 3)$, is the graph on $2k$ vertices obtained
from a clique $c_{1},...,c_{k}$ on $k$ vertices and an independent
set $s_{1},...,s_{k}$ on $k$ vertices. The set of edges is defined
by couples $s_{i}c_{i}$, $s_{i}c_{i+1}$, $1 \leq i < k$, and
$s_{k}c_{k}$, $s_{k}c_{1}$, see figure $1$.

%%%%%%%%%%%%%%%%%%%%%%
%     Figure 1       %
%%%%%%%%%%%%%%%%%%%%%%

\begin{figure}[h!]
%\capstart
\begin{center}
\begin{tikzpicture} [scale=1]

    \filldraw (-5.5+1.45 + 0.2,1.25 + 1.5 + 0.2 - 10.5) circle (2.1pt);
    \filldraw (-5.5 - 1.45 - 0.2,1.25 + 1.5 + 0.2 - 10.5) circle (2.1pt);
    \filldraw (-5.5 + 1.43 + 0.882,1.25 + 0.464 -1.1 - 10.5) circle (2.1pt);
    \filldraw (-5.5 - 1.43 - 0.882,1.25 + 0.464 -1.1 - 10.5) circle (2.1pt);
    \filldraw (-5.5 + 0.2,1.25 - 1.21 -1.2 - 10.5) circle (2.1pt);
    \filldraw (-5.5,1.25 + 1.5 - 10.5) circle (2.1pt);
    \filldraw (-5.5 + 1.43,1.25 + 0.464 - 10.5) circle (2.1pt);
    \filldraw (-5.5 - 1.43,1.25 + 0.464 - 10.5) circle (2.1pt);
    \filldraw (-5.5 + 0.882,1.25 - 1.21 - 10.5) circle (2.1pt);
    \filldraw (-5.5 - 0.882,1.25 - 1.21 - 10.5) circle (2.1pt);

\draw[-,thick] (-5.5,1.25 + 1.5 - 10.5) -- (-5.5 + 1.43,1.25 + 0.464
- 10.5); \draw[-,thick] (-5.5 + 1.43,1.25 + 0.464 - 10.5) -- (-5.5 -
1.43,1.25 + 0.464 - 10.5); \draw[-,thick] (-5.5 - 1.43,1.25 + 0.464
- 10.5) -- (-5.5 + 0.882,1.25 - 1.21 - 10.5); \draw[-,thick] (-5.5 +
0.882,1.25 - 1.21 - 10.5) -- (-5.5 - 0.882,1.25 - 1.21 - 10.5);
\draw[-,thick] (-5.5 - 0.882,1.25 - 1.21 - 10.5) -- (-5.5,1.25 + 1.5
- 10.5);

\draw[-,thick] (-5.5+1.45 + 0.2,1.25 + 1.5 + 0.2 - 10.5)--(-5.5,1.25
+ 1.5 - 10.5); \draw[-,thick] (-5.5+1.45 + 0.2,1.25 + 1.5 + 0.2 -
10.5)--(-5.5 + 1.43,1.25 + 0.464 - 10.5); \draw[-,thick] (-5.5 -
1.45 - 0.2,1.25 + 1.5 + 0.2 - 10.5)--(-5.5,1.25 + 1.5 - 10.5);
\draw[-,thick] (-5.5 - 1.45 - 0.2,1.25 + 1.5 + 0.2 - 10.5)--(-5.5 -
1.43,1.25 + 0.464 - 10.5); \draw[-,thick] (-5.5 + 1.43 + 0.882,1.25
+ 0.464 -1.1 - 10.5)--(-5.5 + 0.882,1.25 - 1.21 - 10.5);
\draw[-,thick] (-5.5 + 1.43 + 0.882,1.25 + 0.464 -1.1 - 10.5)--(-5.5
+ 1.43,1.25 + 0.464 - 10.5); \draw[-,thick] (-5.5 - 1.43 -
0.882,1.25 + 0.464 -1.1 - 10.5)--(-5.5 - 0.882,1.25 - 1.21 - 10.5) ;
\draw[-,thick] (-5.5 - 1.43 - 0.882,1.25 + 0.464 -1.1 - 10.5)--(-5.5
- 1.43,1.25 + 0.464 - 10.5); \draw[-,thick] (-5.5 + 0.2,1.25 - 1.21
-1.2 - 10.5)--(-5.5 + 0.882,1.25 - 1.21 - 10.5); \draw[-,thick]
(-5.5 + 0.2,1.25 - 1.21 -1.2 - 10.5)--(-5.5 - 0.882,1.25 - 1.21 -
10.5);

\draw[-,thick] (-5.5,1.25 + 1.5 - 10.5) -- (-5.5 - 1.43,1.25 + 0.464
- 10.5); \draw[-,thick] (-5.5 - 1.43,1.25 + 0.464 - 10.5) -- (-5.5 -
0.882,1.25 - 1.21 - 10.5); \draw[-,thick] (-5.5 - 0.882,1.25 - 1.21
- 10.5) -- (-5.5 + 1.43,1.25 + 0.464 - 10.5); \draw[-,thick] (-5.5 +
1.43,1.25 + 0.464 - 10.5) -- (-5.5 + 0.882,1.25 - 1.21 - 10.5);
\draw[-,thick] (-5.5 + 0.882,1.25 - 1.21 - 10.5) -- (-5.5,1.25 + 1.5
- 10.5);

%\draw (-0.5,1.25 - 10.25) node {Sun graph}; \draw (-0.5,1.25 -
%10.75); \draw (4.5,1.25 - 10.5) ;

\end{tikzpicture}
\end{center}
\caption{A $k-$Sun graph} \label{Fig1}
\end{figure}
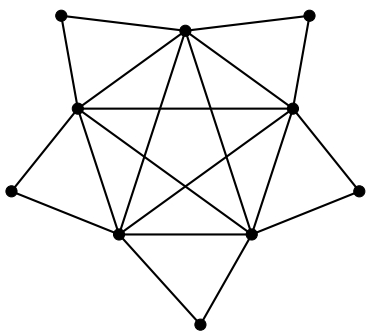

The paper is organized as follow, the Wiener and the Wiener polarity
indices of sun graphs are computed in section $2$. In section $3$, a
relationship between those indices with some other topological
indices are presented. Finally, in section $4$ we find the Hosoya
(Wiener) polynomial for sun graphs.

%%%%%%%%%%%%%%%%%%%%%%%%%%%%%%%%%%%%%%%%%%%%%%%%%%%%%%%%%%%%%%%%%%%
\section{Exact Wiener and Wiener Polarity Index of $k$-sun graphs}%
%%%%%%%%%%%%%%%%%%%%%%%%%%%%%%%%%%%%%%%%%%%%%%%%%%%%%%%%%%%%%%%%%%%

Let $G=(V,E)$ be a $k$-sun graph for $k \geq 3$, obtained from a
clique $c_{1},...,c_{k}$ on $k$ vertices and an independent set
$s_{1},...,s_{k}$.

\begin{theorem}
Let $G$ be a $k-$sun graph ($k\geq 3$) then, $W(G)=k(4k-5)$.
\end{theorem}

\begin{proof}
We have by definition \begin{eqnarray*}W(G)&=&\displaystyle
\frac{1}{2} \sum_{v\in V(G)} d_{G}(v),\\&& =\displaystyle
\frac{1}{2} \sum_{v\in U(G)} d_{G}(v)+ \frac{1}{2} \sum_{v\in C(G)}
d_{G}(v).
\end{eqnarray*}
Where $C(G)$ a clique $c_{1},...,c_{k}$ on $k$ vertices and $U(G)$
an independent set $s_{1},...,s_{k}$ of $G$.

Let $U=\displaystyle \sum_{v\in U(G)} d_{G}(v)$ and $C=\displaystyle
\sum_{v\in C(G)} d_{G}(v)$.
Then, for a vertex $u_{i}$ chosen arbitrarily, then\\
\begin{eqnarray*} U&=&k.(d(u_{i},c_{i})+d(u_{i},c_{i+1})+d(u_{i},u_{i+1})+d(u_{i},u_{i-1})\\&&+(k-2)d(u_{i},c_{j})+(k-3)d(u_{i},u_{j}))\end{eqnarray*}
with $d(u_{i},c_{i})=1$, $d(u_{i},c_{i+1})=1$, $d(u_{i},u_{i+1})=2$,
$d(u_{i},u_{i-1})=2$, $d(u_{i},c_{j})=2$\\ and $d(u_{i},u_{j})=3$.
Hence, $U=k.(5k-7)$.

Moreover, for a vertex $c_{i}$ chosen arbitrarily, we have\\
$C=k.((k-1)d(c_{i},c_{j})+d(c_{i},u_{i})+d(c_{i},u_{i+1})+(k-2)d(c_{i},u_{j}))$\\
with $d(c_{i},c_{j})=1$, $d(c_{i},u_{i})=1$, $d(c_{i},u_{i+1})=1$,
$d(c_{i},u_{j})=2$.\\ Then, $C=k.(3k-3)$, proving a result.

\end{proof}

\begin{theorem}
Let $G$ be a $k-$sun graph then $W_{P}(G)=\displaystyle
\frac{k(k-3)}{2}$.
\end{theorem}

\begin{proof}
If $k=3$, then $W_{p}(G)=0$ since $diam(G)=2$. \\Let $G=(V,E)$ be a
$k$-sun graph for $k \geq 4$. \\Let compute $W_{P}(G)= \displaystyle
\mid \{ \{u, v\} \mid d_{G}(u,v) = 3, u, v \in V \} \mid$.

All distances of all pairs of vertices $\{u,v\}$ such that
$d_{G}(u,v)=3$ are resumed below\\

If $u=c_{i}$ and $v=c_{j}$ then $d_{G}(u,v)\leq 1$ for $1\leq i\leq
k$, $1\leq j\leq k$,\\

If $u=s_{i}$ and $v=c_{i}$ then $d_{G}(u,v)= 1$ for $1\leq i\leq
k$,\\

If $u=s_{i}$ and $v=c_{i+1}$ then $d_{G}(u,v)= 1$ for $1\leq i\leq k
$,\\

$d_{G}(u,v)= 2$, for all $u=s_{i}, v=c_{j}$, $j\neq i+1$, $i\neq k$
and $j\neq 1$, with $1\leq i\leq k$, $1\leq j\leq k$,\\

Finally, $d_{G}(u,v)= 3$ if and only if $u=s_{i}, v=s_{j}$, with
$i\neq j$ $j\neq i+1$, $1\leq i\leq k$, $1\leq j\leq k$. $|\{
\{u,v\} / d_{G}(u,v)= 3\}|=k-3 $, proving the result.

\end{proof}

%%%%%%%%%%%%%%%%%%%%%%%%%%%%%%%%%%%%%%%%%%%%%%%%%%%%%%%%%%%%%%%%%%%%%%%%%
\section{Relation between Wiener and Wiener Polarity index}
%%%%%%%%%%%%%%%%%%%%%%%%%%%%%%%%%%%%%%%%%%%%%%%%%%%%%%%%%%%%%%%%%%%%%%%%%
The first Zagreb index $M_{1}(G)$, is defined as the summation of
squares of the degrees of the vertices, and the second Zagreb index
$M_{2}(G)$, is the sum of the products of the degrees of pairs of
adjacent vertices of the graph $G$. These topological indices were
introduced by Gutman and Trinajsti\'{c} \cite{GT72}.

In \cite{II13} and for $d\geq 1$, authors define a generalization of
Wiener polarity index as the number of unordered pairs of vertices
$\{u,v\}$ of $G$ such that the shortest distance between $u$ and $v$
is $d$.

\begin{center}$W_{d}(G)=\displaystyle \mid \{ \{u, v\} \mid d_{G}(u,v) = d,
u, v \in V \}$. \end{center} So we have \begin{center} $W(G)=\sum
\limits_{d=1}^{diam(G)} W_{d}(G)$.\end{center}

\begin{lemma}\cite{BY11}
Let $G$ be a graph then \begin{center} $W_{p}(G)\leq \displaystyle
\frac{n(n-1)}{2}-\frac{1}{2}M_{1}(G),$ \end{center} with equality if
$diam(G)=3$.
\end{lemma}

\begin{corollary}\cite{BY11}
If $diam(G)\leq 3$ then $W(G) = \displaystyle
\frac{3n(n-1)}{2}-\frac{1}{2}M_{1}(G)-m$.
\end{corollary}

This corollary is valid for sun graphs since their diameter is $\leq
3$ with $n=2k$.

\begin{corollary}
Let $G$ be a sun graph then $W(G) =\displaystyle
\frac{3n(n-1)}{2}-\frac{1}{2}M_{1}(G)-m$.
\end{corollary}

\begin{proposition}
Let $G$ be a sun graph then $W(G) = \displaystyle
n(n-1)+W_{p}(G)-m$.
\end{proposition}

\begin{proof}
From corollary $1,2$ and lemma $1$, we have\\

$M_{1}(G)= \displaystyle n(n-1)+W_{p}(G)-m$,\\

$W(G)= \displaystyle
\frac{3n(n-1)}{2}-\frac{1}{2}n(n-1)+W_{p}(G)-m-m$,\\

hence $W(G) = \displaystyle n(n-1)+W_{p}(G)-m$.

\end{proof}

%%%%%%%%%%%%%%%%%%%%%%%%%%%%%%%%%%%%%%%%%%%%%%%%%%%%%%%%%
\section{The Hosoya (Wiener) polynomial of sun graphs}
%%%%%%%%%%%%%%%%%%%%%%%%%%%%%%%%%%%%%%%%%%%%%%%%%%%%%%%%%
The Hosoya (Wiener) polynomial of graphs is defined by:

\begin{center} $H(G,t)=\displaystyle
\sum_{k=1}^{D}d(G,k)t^{k}$.\end{center}

The wiener index of a graph $G$ can be determined as the first
derivative of the polynomial $H(G,t)$ at $t=1$, such that
$W(G)=\displaystyle \sum_{k=1}^{D}k. t^{k}$

\begin{theorem}
Let $G$ be a sun graph, then we have\\
$H(G,t)= \displaystyle
\frac{1}{2}k(k+3).t+k(k-1).t^{2}+\frac{k(k-3)}{2}.t^{3}$.

\end{theorem}

\begin{proof}

We have $H(G,t)=\displaystyle \sum_{k=1}^{3}d(G,k)t^{k}$, since $G$
is a sun graph.

We have $d(G,1)=|\{(u,v) \ d(u,v)=1\}|$, the pairs are
$\displaystyle \{u_{i},c_{i}\}_{i=1}^{k}$; $\displaystyle
\{c_{i},u_{i+1}\}_{i=1}^{k}$, and $\displaystyle
\{c_{i},c_{j}\}_{i,j=1}^{k}$. Then $d(G,1)= \displaystyle
k+k+\frac{k(k-1)}{2}$.

$d(G,2)=|\{(u,v) \ d(u,v)=2\}|$, the pairs are $\displaystyle
\{u_{i},u_{i+1}\}_{i=1}^{k}$; $\displaystyle
\{u_{i},c_{j}\}_{i=1,i\neq 1,k}^{k}$. Then $d(G,2)=k+k(k-2)=k(k-1)$.

$d(G,3)=|\{(u,v) \ d(u,v)=3\}|$, the pairs are $\displaystyle
\{u_{i},u_{j}\}_{i=1,j\neq i+1,i,i-1}^{k}$. Then $d(G,3)=
\displaystyle \frac{k(k-3)}{2}$, proving the result.

\end{proof}

\begin{remark}
The wiener index of a sun graph $G$ can be calculated in another way
such as the first derivative of the polynomial $H(G,t)$ at $t=1$, then\\
$W(G)= \displaystyle
\frac{1}{2}k(k+3)+2k(k-1)+3\frac{k(k-3)}{2}=k(4k-5)$
\end{remark}

\begin{corollary}
Let $G=S_{3}$ be the $3$-sun graph, then $H(S_{3},t)=9.t+6.t^{2}$.

\end{corollary}

%%%%%%%%%%%%%%%%%%%%%%%%%%%%%%%%%%%%%%%%%%%%%%%%%%%%%%%%%%%%%%%%%%%%%%%%%%%%%%%%%%%%%%%%%%%

\end{document}